\documentclass[oneside,english]{amsart}
\usepackage[T1]{fontenc}
\usepackage[latin9]{inputenc}
\usepackage{amsthm}
\usepackage{amstext}
\usepackage{amssymb}
\usepackage{esint}
\usepackage{amscd}
\usepackage{color}
\makeatletter
\numberwithin{equation}{section}
\numberwithin{figure}{section}
\theoremstyle{plain}
\newtheorem{thm}{\protect\theoremname}[section]

\theoremstyle{plain}
\theoremstyle{definition}

\theoremstyle{plain}

\theoremstyle{plain}

\theoremstyle{plain}
\makeatother

\usepackage{babel}
\providecommand{\definitionname}{Definition}
\providecommand{\lemmaname}{Lemma}
\providecommand{\theoremname}{Theorem}
\providecommand{\corollaryname}{Corollary}
\providecommand{\remarkname}{Remark}
\providecommand{\propositionname}{Proposition}
	
\DeclareMathOperator{\loc}{loc}
\DeclareMathOperator{\dist}{dist}

\DeclareMathOperator*{\esssup}{esssup}

\DeclareMathOperator{\cp}{cap}
\DeclareMathOperator{\ACL}{ACL}
\DeclareMathOperator{\BMO}{BMO}

\begin{document}

\title[Composition operators, $Q$-mappings and weighted Sobolev inequalities]{Composition operators on Sobolev spaces, $Q$-mappings and weighted Sobolev inequalities}

\author{Alexander Menovschikov and Alexander Ukhlov}
\begin{abstract}
In this paper we give connections between mappings which generate bounded composition operators on Sobolev spaces and $Q$-mappings. On this base we obtain measure distortion properties $Q$-homeomorphisms. Using the composition operators on Sobolev spaces we obtain weighted Sobolev inequalities with special weights which are Jacobians of $Q$-mappings.
\end{abstract}
\maketitle
\footnotetext{The first named author was supported by the grant GA\v{C}R 20-19018Y.}

\section{Introduction }

The notion of $Q$-mappings was introduced in \cite{MRSY01} (see also \cite{MRSY04}--\cite{MRSY09}) as a generalization of quasiconformal mappings. The main idea behind such generalization is a connection between quasiconformal mappings theory and Beltrami equations. Recall that a  homeomorphism $\varphi: \Omega\to\widetilde{\Omega}$ of domains
$\Omega,\widetilde{\Omega}\subset \mathbb R^n$ is called a {\it $Q$-homeomorphism}, with a non-negative measurable function $Q$, if
$$
M\left(\varphi \Gamma\right)\leqslant \int\limits_{\Omega} Q(x)\cdot
\rho^{n}(x)dx
$$
for every family $\Gamma$ of rectifiable paths in $\Omega$ and every admissible function $\rho$ for $\Gamma$.

Composition operators on Sobolev spaces arise in \cite{M69,Sl41} and represent the significant part of the geometric analysis of Sobolev spaces due to its applications to the Sobolev embedding theory, that includes embedding theorems, extension theorems and traces theorems \cite{GGu,GS82,GU,GU16_2,KZ20} and to the spectral theory of elliptic operators, see, for example, \cite{GU16,GU17}. 

The geometric theory of composition operators on Sobolev spaces \cite{U93,VU04} is closely connected with various generalizations of quasiconformal mappings which are defined in capacity (moduli) terms and were studied by many authors, see, for example, \cite{HK14,K86,Kr64,MRSY09,P69,R96,S85}. In \cite{MU21}, we consider $BMO$-quasiconformal mappings in connection with composition operators on Hardy--Sobolev spaces. In the present work we consider connections of composition operators on Sobolev spaces with $Q$-homeomorphisms $\varphi:\Omega\to\widetilde{\Omega}$, $Q\in L_1(\Omega)$.
$Q$-mappings with a function $Q$ belongs to the $A_n$-Muckenhoupt class are mappings which inverse to homeomorphisms generating bounded composition operators on the weighted Sobolev spaces \cite{UV08} (see, also \cite{V20}). 

In the present article we prove that in the case $n-1=q<p=n$ mappings generated bounded composition operators on Sobolev spaces are $Q$-homeomorphisms with $Q\in L_1(\Omega)$. Inverse, $Q$-homeomorphisms with $Q\in L_1(\Omega)$ generate a bounded composition operators
$$
\varphi^{\ast}: L^1_n(\widetilde{\Omega})\to L^1_1(\Omega).
$$

In the case $n=2$  we prove the following main result of this article which gives connections between mappings which generate bounded composition operators on Sobolev spaces and $Q$-mappings.

\begin{thm}
\label{main1}
Let $\Omega$ and $\widetilde{\Omega}$ are bounded domains in $\mathbb R^2$. Then $\varphi:\Omega\to\widetilde{\Omega}$ is a $Q$-homeomorphism with $Q\in L_1(\Omega)$, if and only if $\varphi$ generates by the composition rule $\varphi^{\ast}(f)=f\circ\varphi$ a bounded operator
$$
\varphi^{\ast}: L^1_2(\widetilde{\Omega}) \to L^1_1(\Omega).
$$
\end{thm}

Hence, $Q$-homeomorphisms $\varphi: \Omega \to \widetilde\Omega$, where $\Omega$ and $\widetilde{\Omega}$ are bounded domains in $\mathbb R^2$, coincide with mappings of finite distortion \cite{HK14,VGR} such that
$$
\int\limits_{\Omega}\frac{|D\varphi(x)|^2}{|J(x,\varphi)|}~dx<\infty.
$$

This correspondence theorem allows us to establish new measure distortion properties of $Q$-homeomorphisms.

\begin{thm}
\label{inverse} 
Let $\varphi: \Omega \to \widetilde\Omega$ be a $Q$-homeomorphism with $Q \in L_1(\Omega)$. Then 
$$
|\varphi^{-1}(\widetilde{A})| \leq C \int\limits_{\widetilde{A}} \left(H_q(y)\right)^{\frac{n}{n-1}}~dy
$$
for any measurable subset $\widetilde{A}\subset\widetilde{\Omega}$ and, hence, $\varphi^{-1}: \widetilde\Omega \to \Omega$ is absolutely continuous with respect to the Lebesgue measure. In particular, $J(\cdot, \varphi^{-1}) \in L_1(\widetilde\Omega)$.
\end{thm}

On the base of methods suggested in \cite{GGu,GU16}, we give applications of $Q$-homeomorphisms to the weighted Sobolev type embedding theorems:

\begin{thm}\label{sobolevineq}
Let a bounded domain $\widetilde{\Omega}\subset\mathbb R^n$ be such that there exists a weak $(p,q)$-quasiconformal mapping $\varphi: \Omega\to\widetilde{\Omega}$, $1 \leq q \leq p \leq n$, of a bounded $(s,q)$-Sobolev--Poincar\'e domain $\Omega\subset\mathbb R^n$ onto $\widetilde{\Omega}$, $q < s$. Then the weighted Poincar\'e inequality 
$$
\inf\limits_{c\in\mathbb R}\left(\int\limits_{\widetilde{\Omega}}|f(y)-c|^s w(y) \, dy\right)^{\frac{1}{s}}\leq B^w_{s,p}(\widetilde{\Omega})
\left(\int\limits_{\widetilde{\Omega}}|\nabla f(y)|^p \, dy \right)^{\frac{1}{p}} 
$$
holds for any function $f\in W^1_p(\widetilde{\Omega})$ with the weight $w(y)=J_{\varphi^{-1}}(y)$, where $J_{\varphi^{-1}}(y)$ is a volume derivative of the inverse mapping to $\varphi:\Omega\to\widetilde{\Omega}$.
\end{thm}

Note, that weighted Sobolev inequalities with quasiconformal weights were considered, for example, in \cite{GU14,HeK95}.

Composition operators on Sobolev spaces have applications to the spectral theory of elliptic operators, see, for example, \cite{GPU18_3,GU16,GU17}. In some cases the composition operators method allows to obtain better estimates than the classical L.~E.~Payne and H.~F.~Weinberger estimates in convex domains \cite{PW}. In the last section of the present work, as an example of applications $Q$-mappings, we give spectral estimates of the Neumann-Laplace operator in H\"older singular domains.

This paper is organized as follows. Section 2 contains definitions and preliminary results.
In Section 3, we prove a connection between composition operators on Sobolev spaces and $Q$-homeomorphisms and consider its measure properties. Section 4 is devoted to the weighted Sobolev inequalities and in the final section, Section 5, we consider applications of $Q$-homeomorphisms to the spectral theory of elliptic operators.

\section{Sobolev spaces and composition operators}

\subsection{Sobolev spaces}

Let us recall the basic notions of the Sobolev spaces and the change of variable formula.

Let $\Omega$ be an open subset of $\mathbb R^n$. The Sobolev space $W^1_p(\Omega)$, $1\leq p\leq\infty$, is defined \cite{M}
as a Banach space of locally integrable weakly differentiable functions
$f:\Omega\to\mathbb{R}$ equipped with the following norm: 
\[
\|f\mid W^1_p(\Omega)\|=\| f\mid L_p(\Omega)\|+\|\nabla f\mid L_p(\Omega)\|,
\]
where $\nabla f$ is the weak gradient of the function $f$, i.~e. $ \nabla f = (\frac{\partial f}{\partial x_1},...,\frac{\partial f}{\partial x_n})$.

The homogeneous seminormed Sobolev space $L^1_p(\Omega)$, $1\leq p\leq\infty$, is defined as a space
of locally integrable weakly differentiable functions $f:\Omega\to\mathbb{R}$ equipped
with the following seminorm: 
\[
\|f\mid L^1_p(\Omega)\|=\|\nabla f\mid L_p(\Omega)\|.
\]

In the Sobolev spaces theory, a crucial role is played by capacity as an outer measure associated with Sobolev spaces \cite{M}. In accordance to this approach, elements of Sobolev spaces $W^1_p(\Omega)$ are equivalence classes up to a set of $p$-capacity zero \cite{MH72}. The exact definition of capacity will be given below.

The mapping $\varphi:\Omega\to\mathbb{R}^{n}$ belongs to the Sobolev space $W^1_{p,\loc}(\Omega,\mathbb R^n)$, if its coordinate functions belongs to $W^1_{p,\loc}(\Omega)$. In this case, the formal Jacobi matrix $D\varphi(x)$ and its determinant (Jacobian) $J(x,\varphi)$
are well defined at almost all points $x\in\Omega$. The norm $|D\varphi(x)|$ is an operator norm of $D\varphi(x)$,

Let us recall the change of variable formula in the Lebesgue integral \cite{F69, H93}.
Suppose a homeomorphism $\varphi : \Omega\to \mathbb R^n$ be such that
there exists a collection of closed sets $A_k\subset A_{k+1}\subset \Omega$, $k=1,2,...$, for which restrictions $\varphi \vert_{A_k}$ are Lipschitz mappings on the sets $A_k$ and 
$$
\biggl|\Omega\setminus\sum\limits_{k=1}^{\infty}A_k\biggr|=0.
$$
Then there exists a measurable set $S\subset \Omega$, $|S|=0$, such that  the homeomorphism $\varphi:\Omega\setminus S \to \mathbb R^n$ has the Luzin $N$-property (the image of a set of measure zero has measure zero) and the change of variable formula
\begin{equation}
\label{chvf}
\int\limits_E f\circ\varphi (x) |J(x,\varphi)|~dx=\int\limits_{\mathbb R^n\setminus \varphi(S)} f(y)~dy
\end{equation}
holds for every measurable set $E\subset \Omega$ and every non-negative measurable function $f: \mathbb R^n\to\mathbb R$.

Note, that Sobolev homeomorphisms of the class $W^1_{1,\loc}(\Omega)$ satisfy the conditions of the change of variable formula \cite{H93} and, therefore, for Sobolev homeomorphisms the change of variable formula \eqref{chvf} holds.

If the mapping $\varphi$ possesses the Luzin $N$-property, then $|\varphi (S)|=0$ and the second integral can be rewritten as the integral on $\mathbb R^n$. Note, that Sobolev homeomorphisms of the class $L^1_p(\Omega)$, $p\geq n$, possess the Luzin $N$-property \cite{VGR}.

The following formula can be found, for example, in \cite{H50}. Let  $\varphi : \Omega\to \widetilde\Omega$ be a homeomorphism. We define a volume derivative of the inverse mapping as
$$
J_{\varphi^{-1}}(y) := \lim_{r \to 0} \frac{|\varphi^{-1}(B(y,r))|}{|B(y,r)|}.
$$
If $\varphi : \Omega\to \widetilde\Omega$ possesses the Luzin $N^{-1}$-property (the preimage of a set of measure zero has measure zero), then the following change of variables formula holds:
\begin{equation}
\label{chvf2}
\int\limits_{\Omega}f\circ\varphi (x) dx=\int\limits_{\widetilde\Omega} f(y) J_{\varphi^{-1}}(y) dy.
\end{equation}

\subsection{Composition operators on Sobolev spaces}

Let $\Omega$ and $\widetilde{\Omega}$ be domains in the Euclidean space $\mathbb R^n$. Then a homeomorphism $\varphi:\Omega\to\widetilde{\Omega}$ generates a bounded composition
operator 
\[
\varphi^{\ast}:L^1_p(\widetilde{\Omega})\to L^1_q(\Omega),\,\,\,1\leq q\leq p\leq\infty,
\]
by the composition rule $\varphi^{\ast}(f)=f\circ\varphi$, if for
any function $f\in L^1_p(\widetilde{\Omega})$, the composition $\varphi^{\ast}(f)\in L^1_q(\Omega)$
is defined quasi-everywhere in $\Omega$ and there exists a constant $K_{p,q}(\varphi;\Omega)<\infty$ such that 
\[
\|\varphi^{\ast}(f)\mid L^1_q(\Omega)\|\leq K_{p,q}(\varphi;\Omega)\|f\mid L^1_p(\widetilde{\Omega})\|.
\]

Recall that the $p$-dilatation \cite{Ger69} of a Sobolev mapping $\varphi: \Omega\to \widetilde{\Omega}$ at a point $x\in\Omega$ is defined as
$$
K_p(x)=\inf \{k(x): |D\varphi(x)|\leq k(x) |J(x,\varphi)|^{\frac{1}{p}}\}.
$$

The following theorem gives a characterization of composition operators in terms of integral characteristics of mappings of finite distortion. Recall that a weakly differentiable mapping $\varphi:\Omega\to\mathbb{R}^{n}$ is a mapping of finite distortion if $D\varphi(x)=0$ for almost all $x$ from $Z=\{x\in\Omega: J(x,\varphi)=0\}$ \cite{VGR}. 

\begin{thm}
\label{CompTh} A homeomorphism $\varphi:\Omega\to\widetilde{\Omega}$
between two domains $\Omega$ and $\widetilde{\Omega}$ generate a bounded composition
operator 
\[
\varphi^{\ast}:L^1_p(\widetilde{\Omega})\to L^1_{q}(\Omega),\,\,\,1\leq q\leq p\leq\infty,
\]
 if and only if $\varphi\in W^1_{q,\loc}(\Omega)$, has finite distortion in the case $p<\infty$,
and 
\[
K_{p,q}(\varphi;\Omega) := \|K_p \mid L_{\kappa}(\Omega)\|<\infty, \,\,1/q-1/p=1/{\kappa}\,\,(\kappa=\infty, \text{ if } p=q).
\]
The norm of the operator $\varphi^\ast$ is estimated as $\|\varphi^\ast\| \leq K_{p,q}(\varphi;\Omega)$.
\end{thm}

This theorem was proved in \cite{U93,VU02} (see, also \cite{GGR95,V88} for the case $p=q$ ) and the limit case $p=\infty$ was considered in \cite{GU10}. Note, that in the case $p=\infty$ we have that
$$
K_{\infty,q}(\varphi;\Omega) :=\|\varphi\mid L^1_q(\Omega)\|.
$$

{\it Let us recall, that homeomorphisms $\varphi:\Omega\to\widetilde{\Omega}$ which satisfy conditions of Theorem~\ref{CompTh}  are called as weak $(p,q)$-quasiconformal mappings \cite{GGR95,VU98}. }

In the case of weak $(p,q)$-quasiconformal mappings, the following composition duality theorem holds \cite{U93} (the detailed proof can be found in  \cite{GU19}):

\begin{thm}
\label{CompThD} Let a homeomorphism $\varphi:\Omega\to\widetilde{\Omega}$
between two domains $\Omega$ and $\widetilde{\Omega}$ generate a bounded composition
operator 
\[
\varphi^{\ast}:L^1_p(\widetilde{\Omega})\to L^1_{q}(\Omega),\,\,\,n-1<q \leq p< \infty,
\]
then the inverse mapping $\varphi^{-1}:\widetilde{\Omega}\to\Omega$ generates a bounded composition operator 
\[
\left(\varphi^{-1}\right)^{\ast}:L^1_{q'}(\Omega)\to L^1_{p'}(\widetilde{\Omega}),
\]
where $p'=p/(p-n+1)$, $q'=q/(q-n+1)$. 
\end{thm}

In the case $\Omega, \widetilde{\Omega}\subset\mathbb R^2$ the following generalized statement of Theorem~\ref{CompThD} \cite{U93} is correct:

\begin{thm}
\label{CompThDP} Let a homeomorphism $\varphi:\Omega\to\widetilde{\Omega}$
between two domains $\Omega, \widetilde{\Omega}\subset\mathbb R^2$ generates a bounded composition
operator 
\[
\varphi^{\ast}:L^1_p(\widetilde{\Omega})\to L^1_{q}(\Omega),\,\,\,1\leq q \leq p< \infty,
\]
then the inverse mapping $\varphi^{-1}:\widetilde{\Omega}\to\Omega$ generates a bounded composition operator 
\[
\left(\varphi^{-1}\right)^{\ast}:L^1_{q'}(\Omega)\to L^1_{p'}(\widetilde{\Omega}),\,\,\frac{1}{p}+\frac{1}{p'}=1, \frac{1}{q}+\frac{1}{q'}=1.
\]
In the case $1< q \leq p< \infty$ the inverse assertion is correct also and $p''=p$, $q''=q$.
\end{thm}

Note that in the case $q=1$ the inverse assertion will be correct with additional assumptions of finite distortion and Luzin $N$-property. 

\section{Composition operators and $Q$-homeomorphisms}

\subsection{Modulus and capacity}

The theory of $Q$-quasiconformal mappings has been extensively developed in recent decades, see, for example,  \cite{MRSY09}. Let us give the  basic definitions.

The linear integral is denoted by
$$
\int\limits_{\gamma}\rho~ds=\sup\int\limits_{\gamma'}\rho~ds=\sup\int\limits_0^{l(\gamma')}\rho(\gamma'(s))~ds
$$
where the supremum is taken over all closed parts $\gamma'$ of $\gamma$ and $l(\gamma')$ is the length of $\gamma'$. Let $\Gamma$ be a family of curves in $\mathbb R^n$. Denote by $adm(\Gamma)$ the set of Borel functions (admissible functions)
$\rho: \mathbb R^n\to[0,\infty]$ such that the inequality
$$
\int\limits_{\gamma}\rho~ds\geqslant 1
$$
holds for locally rectifiable curves $\gamma\in\Gamma$.

Let $\Gamma$ be a family of curves in $\overline{\mathbb R^n}$, where $\overline{\mathbb R^n}$ is a one point compactification of the Euclidean space $\mathbb R^n$. The quantity
$$
M(\Gamma)=\inf\int\limits_{\mathbb R^n}\rho^{n}~dx
$$
is called the (conformal) module of the family of curves $\Gamma$ \cite{MRSY09}. The infimum is taken over all admissible functions
$\rho\in adm(\Gamma)$.

Let $\Omega$ be a bounded domain in $\mathbb R^n$ and $F_0, F_1$ be disjoint non-empty compact sets in the
closure of $\Omega$. Let $M(\Gamma(F_0,F_1;\Omega))$ stand for the
module of a family of curves which connect $F_0$ and $F_1$ in $\Omega$. Then \cite{MRSY09}
\begin{equation}\label{eq2}
M(\Gamma(F_0,F_1;\Omega)) = \cp_{n}(F_0,F_1;\Omega)\,,
\end{equation}
where $\cp_{n}(F_0,F_1;\Omega)$ is a conformal capacity of the condenser $(F_0,F_1;\Omega)$ \cite{M}.

Recall that a homeomorphism $\varphi: \Omega\to\widetilde{\Omega}$ of domains
$\Omega,\widetilde{\Omega}\subset \mathbb R^n$ is called a $Q$-homeomorphism \cite{MRSY09} with a non-negative measurable function $Q$ if
$$
M\left(\varphi \Gamma\right)\leqslant \int\limits_{\Omega} Q(x)\cdot
\rho^{n}(x)dx
$$
for every family $\Gamma$ of rectifiable paths in $\Omega$ and every admissible function $\rho$ for $\Gamma$.

\subsection{Mappings of finite distortion}

Let us define two distortion functions for Sobolev mappings of finite distortion $\varphi: \Omega \to \widetilde\Omega$.

\noindent
The outer  (quasi)conformal dilatation
$$
K^O(x,\varphi)=
\begin{cases}
\frac{|D\varphi(x)|^n}{|J(x,\varphi)|},& \,\, J(x,\varphi)\ne 0,\\
0,& \,\, J(x,\varphi)= 0.
\end{cases}
$$
The inner (quasi)conformal dilatation
$$
K^I(x,\varphi)=
\begin{cases}
\frac{|J(x,\varphi)|}{l(D\varphi(x))^n},& \,\, J(x,\varphi)\ne 0,\\
0,& \,\, J(x,\varphi)= 0,
\end{cases}
$$
where $l(D\varphi(x)) = \min\{ |D\varphi(x)h| : h \in \mathbb{R}^n, |h|=1 \}$.

Note that $K^I(x) \leq (K^O(x))^{n-1}$ and $K^O(x) \leq (K^I(x))^{n-1}$. 

In this section we consider $Q$-homeomorphisms with $Q \in L_s(\Omega)$, $s\geq 1$. 
In \cite{MRSY09}, it was proven that $Q$-homeomorphisms $\varphi: \Omega \to \widetilde\Omega$ with $Q \in L_{1, \loc}(\Omega)$ belong to the Sobolev space $W^1_{1,\loc}(\Omega)$, differentiable a.e., have a finite distortion  and possess the Luzin-$N^{-1}$ property. 

\begin{thm}\label{Q-comp}
Let $\varphi: \Omega \to \widetilde{\Omega}$ be a $Q$-homeomorphism with $Q \in L_s(\Omega)$. Then $\varphi$ generates the bounded composition operator 
$$
\varphi^{\ast}: L^1_n(\widetilde\Omega) \to L^1_{q}(\Omega),
$$
where $q={sn}/{(s+n-1)}$.
In particular, if $Q \in L_1(\Omega)$, then $\varphi$ generates the bounded composition operator 
$$
\varphi^{\ast}: L^1_n(\widetilde\Omega) \to L^1_1(\Omega).
$$
\end{thm}

\begin{proof}
Because $\varphi$ is a $Q$-homeomorphism with $Q \in L_s(\Omega)$, then $Q\in L_{1,\loc}(\Omega)$. Hence 
$\varphi$ belongs to the Sobolev space $W^1_{1,\loc}(\Omega)$ and has a finite distortion \cite{MRSY09}.

In \cite{MRSY09} it was proved that for $Q\in L_{1,\loc}(\Omega)$ the inequality 
\begin{equation}
\label{Q-dist}
|D\varphi(x)| \leq C(n) |J(x,\varphi)|^{\frac{1}{n}}Q^{\frac{n-1}{n}}(x)
\end{equation}
holds for almost all $x\in\Omega$. Hence
$$
\left(\frac{|D\varphi(x)|^n}{|J(x,\varphi)|}\right)^{\frac{1}{n-1}}\leq Q(x)\,\,\text{for almost all}\,\,x\in\Omega.
$$

Since $Q \in L_s(\Omega)$, we have 
$$
\int\limits_{\Omega}\left(\frac{|D\varphi(x)|^n}{|J(x,\varphi)|}\right)^{\frac{s}{n-1}}~dx\leq \int\limits_{\Omega}Q^s(x)~dx<\infty.
$$
Then, by Theorem~\ref{CompTh} the mapping $\varphi$ generates the bounded composition operator $\varphi^{\ast}: L^1_n(\widetilde\Omega) \to L^1_{q}(\Omega)$, $q={sn}/{(s+n-1)}$.
\end{proof}

Using the composition duality property in the case of planar domains $\Omega,\widetilde{\Omega}\subset\mathbb R^2$ we obtain:

\begin{thm}
Let $\varphi: \Omega \to \widetilde{\Omega}$ where $\Omega,\widetilde{\Omega}\subset\mathbb R^2$, be a $Q$-homeomorphism with $Q \in L_1(\Omega)$. Then the inverse mapping $\varphi^{-1}:\widetilde{\Omega}\to \Omega$ generates the bounded composition operator 
$$
\left(\varphi^{-1}\right)^{\ast}: L^1_{\infty}(\Omega) \to L^1_{2}(\widetilde\Omega).
$$
In particular, $\varphi^{-1}\in L^1_2(\widetilde\Omega)$.
\end{thm}

\begin{proof}
Since $\varphi: \Omega \to \widetilde{\Omega}$ is a $Q$-homeomorphism with $Q \in L_1(\Omega)$, then by Theorem~\ref{Q-comp} the mapping $\varphi$ generates the bounded composition operator 
$$
\varphi^{\ast}: L^1_2(\widetilde{\Omega})\to L^1_1(\Omega).
$$
Hence by Theorem~\ref{CompThDP} the inverse mapping $\varphi^{-1}:\widetilde{\Omega}\to \Omega$ generates the bounded composition operator 
$$
\left(\varphi^{-1}\right)^{\ast}: L^1_{\infty}(\Omega) \to L^1_{2}(\widetilde\Omega).
$$
In particular, by \cite{GU10}, $\varphi^{-1}\in L^1_2(\widetilde\Omega)$.
\end{proof}

Now we prove the inverse property. 

\begin{thm}\label{comp-Q}
Let a homeomorphism $\varphi: \Omega \to \widetilde{\Omega}$, $\Omega$ and $\widetilde{\Omega}$ are domains in $\mathbb R^n$, $n\geq 3$, satisfy Luzin $N$-property and generate the bounded composition operator 
$$
\varphi^{\ast}: L^1_n(\widetilde\Omega) \to L^1_{n-1}(\Omega).
$$
Then $\varphi$ is a $Q$-homeomorphism with $Q(x)=K^I(x,\varphi)\in L_1(\Omega)$.
\end{thm}

\begin{proof}  
Since $\varphi$ generates the bounded composition operator 
$$
\varphi^{\ast}: L^1_n(\widetilde\Omega) \to L^1_{n-1}(\Omega),
$$
then by Theorem \ref{CompTh}
$$
\int\limits_{\Omega}\left(\frac{|D\varphi(x)|^n}{|J(x,\varphi)|}\right)^{n-1}~dx<\infty.
$$
Hence
$$
\int\limits_{\Omega}K^I(x,\varphi)~dx\leq  
\int\limits_{\Omega}\left(K^O(x,\varphi)\right)^{n-1}~dx= \int\limits_{\Omega}\left(\frac{|D\varphi(x)|^n}{|J(x,\varphi)|}\right)^{n-1}~dx<\infty.
$$

Now by the definition of the $Q$-homeomorphism, we have to show, that for every family $\Gamma$ of paths in $\Omega$ and every $\rho \in adm(\Gamma)$
$$
M(\varphi\Gamma) \leq \int\limits_\Omega K^I(x,\varphi) \rho^n(x) \, dx.
$$

Because $\varphi$ generates the bounded composition operator 
$$
\varphi^{\ast}: L^1_n(\widetilde\Omega) \to L^1_{n-1}(\Omega),
$$
then by Theorem \ref{CompTh}, $\varphi\in W^1_{n-1,\loc}(\Omega)$ and has a finite distortion. Hence by \cite[Theorem B]{GU10} (see, also  \cite[Theorem 1.2]{CHM10}) we obtain that $\varphi^{-1} \in W^1_{1,\loc}(\widetilde\Omega)$ and has a finite distortion.
Now by the change of variables formula 
$$
\int\limits_{\widetilde\Omega}|D\varphi^{-1}(y)|^n~dy=\int\limits_{\Omega}|D\varphi^{-1}(\varphi(x))|^n|J(x,\varphi)|~dx=
\int\limits_{\Omega}\frac{|J(x,\varphi)|}{l(D\varphi(x))^n}<\infty.
$$

It implies, that $\varphi^{-1} \in \ACL^n(\widetilde\Omega)$, is differentiable a.e. in $\widetilde\Omega$ and satisfies the Luzin $N$-property

By Fuglede's theorem (\cite{V71}, p.~95), if $\widetilde\Gamma$ is the family of all paths $\gamma \in \varphi\Gamma$ for which $\varphi^{-1}$ is absolutely continuous on all closed subpaths of $\gamma$, then $M(\varphi\Gamma) = M(\widetilde\Gamma)$. Then, for given $\rho \in adm(\Gamma)$, one consider
$$
\widetilde\rho(y) = 
\begin{cases}
\rho(\varphi^{-1}(y))|D\varphi^{-1}(y)|,& \,\, y \in \widetilde\Omega,\\
0,& \,\, \text{otherwise}.
\end{cases}
$$ 
Then, for $\widetilde\gamma \in \widetilde\Gamma$
$$
\int\limits_{\widetilde\gamma}\widetilde\rho \, ds \geq \int\limits_{\varphi^{-1}\circ\widetilde\gamma} \rho \, ds \geq 1,
$$
and consequently $\widetilde\rho \in adm(\widetilde\Gamma)$.

By the change of variable formula, we obtain
\begin{multline*}
M(\varphi\Gamma) = M(\widetilde\Gamma) \leq \int\limits_{\widetilde\Omega} \widetilde\rho^n \, dy \\
= \int\limits_{\widetilde\Omega} \rho^n(\varphi^{-1}(y))|D\varphi^{-1}(y)|^n \, dy = \int\limits_{\widetilde\Omega} \frac{\rho^n(\varphi^{-1}(y))}{l(D\varphi(\varphi^{-1}(y)))^n} \, dy \\
= \int\limits_{\widetilde\Omega} \rho^n(\varphi^{-1}(y)) K^I(\varphi^{-1}(y), \varphi) J_{\varphi^{-1}}(y) \, dy \leq \int\limits_\Omega K^I(x, \varphi) \rho^n(x) \, dx
\end{multline*}
which completes the proof.
\end{proof}

\subsection{The proof of Theorem~\ref{main1}}

Now we proof Theorem~\ref{main1} which provides the equivalence between weak quasiconformal mappings introduced in \cite{U93} and $Q$-homeomorphisms introduced in \cite{MRSY01} in the case of domains $\Omega,\widetilde{\Omega}\subset \mathbb R^2$.
Let us recall the statement of the theorem.

\vskip 0.2cm

\noindent
{\bf Theorem~\ref{main1}.}
{\it Let $\varphi: \Omega \to \widetilde{\Omega}$ be a homeomorphism of planar domains $\Omega,\widetilde{\Omega}\subset\mathbb R^2$. Then $\varphi$ be a $Q$-homeomorphism with $Q \in L_1(\Omega)$ if and only if $\varphi$ generates a bounded composition operator 
$$
\varphi^{\ast}: L^1_2(\widetilde\Omega) \to L^1_{1}(\Omega).
$$
}

\begin{proof} 
Let us first assume, that $\varphi$ is a $Q$-homeomorphism with $Q \in L_1(\Omega)$. Then by Theorem~\ref{Q-comp} the mapping $\varphi$ 
generates a bounded composition operator 
$$
\varphi^{\ast}: L^1_2(\widetilde\Omega) \to L^1_{1}(\Omega).
$$

Now assume, that $\varphi: \Omega \to \widetilde{\Omega}$ generates a bounded composition operator 
$$
\varphi^{\ast}: L^1_2(\widetilde\Omega) \to L^1_{1}(\Omega).
$$
Then by Theorem~\ref{CompThDP} the inverse mapping $\varphi^{-1}$ generates a composition operator 
$$
(\varphi^{-1})^\ast: L^1_\infty(\Omega) \to L^1_2(\widetilde\Omega),
$$ 
and hence, $\varphi^{-1} \in W^1_{2}(\widetilde\Omega)$.

Therefore, $\varphi^{-1} \in \ACL^2(\widetilde\Omega)$, is differentiable a.e. in $\widetilde\Omega$ and satisfies the Luzin $N$-property. Than we can proceed with the same arguments as in previous theorem and conclude, that for every family $\Gamma$ of paths in $\Omega$ and every $\rho \in adm(\Gamma)$
$$
M(\varphi\Gamma) \leq \int\limits_\Omega K^I(x,\varphi) \rho^2(x) \, dx.
$$

Since $\varphi$ generates a bounded composition operator 
$$
\varphi^{\ast}: L^1_2(\widetilde\Omega) \to L^1_{1}(\Omega),
$$
by Theorem~\ref{CompTh}
$$
\int\limits_\Omega K^I(x, \varphi)~dx=\int\limits_\Omega K^O(x, \varphi)~dx=\int\limits_\Omega \frac{|D\varphi(x)|^2}{|J(x,\varphi)|}~dx<\infty,
$$
that completes the proof.
\end{proof}

\subsection{The measure distortion properties of $Q$-mappings}

In this section we establish measure distortion properties of $Q$-mappings on the base of its connection with composition operators on Sobolev spaces. Let a homeomorphism $\varphi:\Omega\to\widetilde{\Omega}$	belongs to the Sobolev space $W^1_{1,\loc}(\Omega)$. We define the following distortion function \cite{VU02}:
$$
H_q(y)=
\begin{cases}
\left(\frac{|D\varphi(x)|^q}{|J(x,\varphi)|}\right)^{\frac{1}{q}},\,\,&x=\varphi^{-1}(y)\in \Omega\setminus (S\cup Z), \\
0,\,\,&x=\varphi^{-1}(y)\in S\cup Z,
\end{cases}
$$
where $S$ is a set from the change of variables formula~(\ref{chvf}) and $Z=\{x\in\Omega: J(x,\varphi)=0\}$.

In \cite{VU02} it was proved:

\begin{thm}
\label{CompTh02} A homeomorphism $\varphi:\Omega\to\widetilde{\Omega}$
between two domains $\Omega$ and $\widetilde{\Omega}$ generate a bounded composition
operator 
\[
\varphi^{\ast}:L^1_p(\widetilde{\Omega})\to L^1_{q}(\Omega),\,\,\,1\leq q\leq p<\infty,
\]
if and only if $\varphi\in W^1_{q,\loc}(\Omega)$, has finite distortion and 
\[
H_{p,q}(\varphi;\widetilde{\Omega}) := \|H_q \mid L_{\kappa}(\widetilde{\Omega})\|<\infty, \,\,1/q-1/p=1/{\kappa}\,\,(\kappa=\infty, \text{ if } p=q).
\]
The norm of the operator $\varphi^\ast$ is estimated as $\|\varphi^\ast\| \leq H_{p,q}(\varphi;\widetilde{\Omega})$.
\end{thm}

The following theorem is a reformulation of a measure distortion property \cite{VU02}.

\begin{thm}
\label{measure}
Let $\varphi:\Omega\to\widetilde{\Omega}$ be a weak $(n,q)$-quasiconformal mapping, $1 \leq q \leq n$. Then for any measurable subset $\widetilde{A}\subset\widetilde{\Omega}$ the following inequality 
$$
|\varphi^{-1}(\widetilde{A})| \leq C \int\limits_{\widetilde{A}} \left(H_q(y)\right)^{\frac{nq}{n-q}}~dy
$$
holds.
\end{thm}

As a consequence of the above theorem we obtain the following result for integrability of the Jacobian (the case of weak $(p,q)$-quasiconformal mappings with $1< q\leq p<n$ was proved in \cite{VU02} by another method).
\begin{thm}
\label{int}
Let $\varphi:\Omega\to\widetilde{\Omega}$ be a weak $(n,q)$-quasiconformal mapping, $1\leq q\leq n$. Then 
$$
J(\cdot, \varphi^{-1}) = J_{\varphi^{-1}}(\cdot) \in L_1(\widetilde\Omega).
$$
\end{thm}

\begin{proof}
Let us fix a point $y_0$ and a ball $B(y_0,r) \subset \widetilde\Omega$. Then by Theorem \ref{measure} we obtain
$$
|\varphi^{-1}(B(y_0,r))|\leq C \int\limits_{B(y_0,r)} \left(H_q(y)\right)^{\frac{nq}{n-q}}~dy.
$$
Divide both sides on $|B(y_0,r)|$
$$
\frac{|\varphi^{-1}(B(y_0,r))|}{|B(y_0,r)|}\leq C \frac{1}{|B(y_0,r)|}\int\limits_{B(y_0,r)} \left(H_q(y)\right)^{\frac{nq}{n-q}}~dy	.
$$
Taking the limit $r \to 0$ by the Lebesgue differentiation theorem (see, for example \cite{H50, F69}) we obtain 
$$
J_{\varphi^{-1}}(y_0) \leq C \left(H_q(y_0)\right)^{\frac{nq}{n-q}},\,\,\text{for almost all}\,\,y_0\in\widetilde{\Omega}.
$$
Hence, integrating over $\widetilde\Omega$, due to the fact, that our mapping is a weak $(n,q)$-quasiconformal,
$$
\int\limits_{\widetilde\Omega} J_{\varphi^{-1}}(y) \, dy \leq C \int\limits_{\widetilde\Omega} \left(H_q(y)\right)^{\frac{nq}{n-q}} \, dy < \infty,
$$
and so
$$
J(\cdot, \varphi^{-1}) = J_{\varphi^{-1}}(\cdot) \in L_1(\widetilde\Omega).
$$
\end{proof}

\noindent
{\bf The proof of Theorem~\ref{inverse}} now follows from Theorem~\ref{measure}, Theorem~\ref{int} and Theorem \ref{Q-comp}.

\section{Poincar\'e inequalities}

In this section we prove the weighted Sobolev--Poincar\'e inequality  and on this base we provide estimate of the exact constant in Sobolev--Poincar\'e inequality.

Let us recall that a bounded domain $\Omega\subset\mathbb R^n$ is said to be $(s,q)$-Sobolev--Poincar\'e domain (see, for example, \cite{GU,GU16}), $1 \leq q,s \leq \infty$, if the following Sobolev--Poincar\'e inequality
$$
\inf\limits_{c\in\mathbb R} \|g-c \mid L_s(\Omega)\| \leq B_{s,q}(\Omega)\|g \mid L^1_q(\Omega)\|
$$
holds for any $g \in L^1_q(\Omega)$ with the constant $B_{s,q}(\Omega)<\infty$. 

Let us recall the statement of the theorem, which is corresponding to the weighted Poincar\'e inequality.

\vskip 0.2cm
\noindent

{\bf Theorem~\ref{sobolevineq}.}
{\it Let a bounded domain $\widetilde{\Omega}\subset\mathbb R^n$ be such that there exists a weak $(p,q)$-quasiconformal mapping $\varphi: \Omega\to\widetilde{\Omega}$, $1 \leq q \leq p \leq n$, of a bounded $(s,q)$-Sobolev--Poincar\'e domain $\Omega\subset\mathbb R^n$ onto $\widetilde{\Omega}$, $q < s$. Then the weighted Poincar\'e inequality 
$$
\inf\limits_{c\in\mathbb R}\left(\int\limits_{\widetilde{\Omega}}|f(y)-c|^s w(y) \, dy\right)^{\frac{1}{s}}\leq B^w_{s,p}(\widetilde{\Omega})
\left(\int\limits_{\widetilde{\Omega}}|\nabla f(y)|^p \, dy \right)^{\frac{1}{p}} 
$$
holds for any function $f\in W^1_p(\widetilde{\Omega})$ with the weight $w(y)=J_{\varphi^{-1}}(y)$, where $J_{\varphi^{-1}}(y)$ is a volume derivative of the inverse mapping to $\varphi:\Omega\to\widetilde{\Omega}$.
}

\begin{proof}

Since the mapping $\varphi:\Omega\to\widetilde{\Omega}$ is a weak $(p,q)$-quasiconformal, then the composition operator $\varphi^\ast: L^1_p(\widetilde\Omega) \to L^1_q(\Omega)$ is bounded and, for a function $f \in L^1_p(\widetilde\Omega)$, the composition $f \circ \varphi$ belongs to $L^1_q(\Omega)$. Morower, from \cite{VU98,VU02} we know, that a weak $(p,q)$-quasiconformal mappings possess the Luzin-$N^{-1}$ property, if $1 \leq q \leq p \leq n$.
 
By the change of variable formula \eqref{chvf2} and since $\Omega$ is a bounded $(s,q)$-Sobolev--Poincar\'e domain, we have 
\begin{multline*}
\inf\limits_{c\in\mathbb R}\left(\int\limits_{\widetilde{\Omega}}|f(y)-c|^s J_{\varphi^{-1}}(y)\,dy\right)^{\frac{1}{s}}
= \inf\limits_{c\in\mathbb R}\left(\int\limits_{\Omega}|f(\varphi(x))-c|^s \, dx\right)^{\frac{1}{s}}\\
\leq B_{s,q} (\Omega) \left(\int\limits_\Omega |\nabla f(\varphi(x))|^q \, dx \right)^{\frac{1}{q}},
\end{multline*}
where $B_{s,q}(\Omega)$ is a best constant in the $(s,q)$-Sobolev--Poincar\'e inequality.

By Theorem \ref{CompTh}
$$
\|\varphi^\ast (f) \mid L^1_q(\Omega)\| \leq K_{p,q} (\Omega; \varphi) \|f \mid L^1_p(\widetilde\Omega)\|.
$$
 Hence
\begin{multline*}
\inf\limits_{c\in\mathbb R}\left(\int\limits_{\widetilde{\Omega}}|f(y)-c|^s J_{\varphi^{-1}}(y)\,dy\right)^{\frac{1}{s}} \leq B_{s,q} (\Omega)\left(\int\limits_\Omega |\nabla f(\varphi(x))|^q \, dx \right)^{\frac{1}{q}} \\
\leq B_{s,q}(\Omega)K_{p,q}(\Omega; \varphi) \left(\int\limits_{\widetilde\Omega} |\nabla f(y)|^p \, dy \right)^{\frac{1}{p}}=
B^w_{s,p}(\widetilde{\Omega})
\left(\int\limits_{\widetilde{\Omega}}|\nabla f(y)|^p \, dy \right)^{\frac{1}{p}}.
\end{multline*}
\end{proof} 

Applications of $Q$-homeomorphisms to the spectral theory of two-dimensional elliptic operators are based on estimates of the constants in non-weighted case $B_{s,p}(\Omega)$ for $\Omega \subset \mathbb{R}^2$ . In \cite{GPU18}, for $\Omega \subset \mathbb R^2$ there was proved the next version of Poincar\'e inequality with the upper estimates of the sharp constant:

\begin{thm}[\cite{GPU18}]\label{exsobolev}
Let $f \in W^1_1(\Omega)$, $\Omega \subset \mathbb R^2$. Then for any $r > 0$ and any $z_0 \in \Omega$, such that $\dist(z_0, \partial\Omega) > 2r$, the following inequality holds:
\begin{equation}\label{PI-Disk}
\left( \iint\limits_{D(z_0,r)} |f(z)-f_{D(z_0,r)}|^2 \, dxdy \right)^{\frac{1}{2}} \leq \frac{3\sqrt{\pi^3}}{4} \iint\limits_{D(z_0,r)} |\nabla f(z)| \, dxdy.
\end{equation}
\end{thm}

Using the bi-Lipschitz change of variables, we can extend this result from a disk to a general domain with Lipschitz boundary and obtain an estimate for the constant $B_{2,2}(\Omega)$.

\begin{thm}
Let there exist a bi-Lipschitz mapping $\varphi: D(0,r) \to \Omega$ and $f \in W^1_1(\Omega)$. Then the following inequality holds:
\begin{equation}\label{PI-Lip}
\left( \iint\limits_{\Omega} |f(x,y)-f_{\Omega}|^2 \, dxdy \right)^{\frac{1}{2}} \leq \frac{3\sqrt{L^5\pi^3}}{4} \iint\limits_{\Omega} |\nabla f(x,y)| \, dxdy,
\end{equation}
where $L$ is a Lipschitz constant for mapping $\varphi$. 
\end{thm}
\begin{proof}
For the simplicity, we prove the inequality for the case $f_\Omega = 0$. Since the function $f$ belongs to $L^1_1(\Omega)$, composition $f\circ \varphi$ with bi-Lipschitz mapping belongs $L^1_1$ and we can apply the change of variable formula.

\begin{multline*}
\left( \iint\limits_{\Omega} |f(u,v)|^2 \, dudv \right)^{\frac{1}{2}} = \left( \iint\limits_{D(0,r)} |f(\varphi(x,y))|^2 |J(x,y, \varphi)| \, dxdy \right)^{\frac{1}{2}} \\
\leq \esssup\limits_{(x,y) \in D(0,r)} |J(x,y, \varphi)|^{\frac{1}{2}} \left( \iint\limits_{D(0,r)} |f(\varphi(x,y))|^2 \, dxdy \right)^{\frac{1}{2}} \\
\end{multline*}

By Theorem \ref{exsobolev}, we obtain the following estimate:

$$
\left( \iint\limits_{\Omega} |f(u,v)|^2 \, dudv \right)^{\frac{1}{2}} \leq \frac{3\sqrt{\pi^3}}{4} \esssup\limits_{(x,y) \in D(0,r)} |J(x,y, \varphi)|^{\frac{1}{2}} \iint\limits_{D(0,r)} |\nabla f(\varphi(x,y))| \, dxdy.
$$

Using the chain rule and the change of variable formula the second time, we infer

\begin{multline*}
\left( \iint\limits_{\Omega} |f(u,v)|^2 \, dudv \right)^{\frac{1}{2}} \\
\leq \frac{3\sqrt{\pi^3}}{4} \esssup\limits_{(x,y) \in D(0,r)} |J(x,y, \varphi)|^{\frac{1}{2}} \esssup\limits_{(x,y) \in D(0,r)} \frac{|D\varphi(x,y)|}{|J(x,y, \varphi)|} \iint\limits_{\Omega} |\nabla f(u,v)| \, dudv.
\end{multline*}

By the Hadamar inequality, we can choose a constant $L$ such that $0 < L^{-1} \leq |J_\varphi(x,y)|^{1/2} \leq |D\varphi(x,y)| \leq L$. Then
$$
\left( \iint\limits_{\Omega} |f(u,v)|^2 \, dudv \right)^{\frac{1}{2}} \leq \frac{3\sqrt{L^5\pi^3}}{4} \iint\limits_{\Omega} |\nabla f(u,v)| \, dudv.
$$
\end{proof}

\section{Spectral estimates of elliptic operators}

In this section we consider applications of $Q$-homeomorphisms to spectral estimates of the Laplace operator in H\"older cusp domains. The detailed description of the spectral theory of non-linear elliptic operators and applications of geometric theory of composition operators see, for example, \cite{GPU18_3,GU16,GU17}. Let us consider the Neumann spectral problem
$$
\begin{cases}
-\Delta u=\mu u \,\,\text{in}\,\,\widetilde\Omega,\\
\frac{\partial u}{\partial\nu}=0 \,\,\,\,\,\,\,\,\text{on}\,\,\partial\widetilde\Omega.
\end{cases}
$$
By the Min-Max Principle the first non-trivial Neumann eigenvalues in the domains $\widetilde\Omega$ can be characterized as 
$$
\mu_1 (\widetilde\Omega) = \left({B_{2,2}(\widetilde\Omega)}\right)^{-2}.
$$

With the help of $Q$-homeomorphism we can estimate the first eigenvalue of the Neumann-Laplacian in the case of cusp domain. The estimate is based on the following anti-commutative diagram \cite{GGu,GU}:

$$
\begin{CD}
L^1_2(\widetilde\Omega) @>{\varphi^\ast}>> L^1_1(\Omega) \\
@VVV @VVV \\
L_2(\widetilde\Omega) @<{(\varphi^{-1})^\ast}<< L_2(\Omega)
\end{CD}
$$

By this diagram, $B_{2,2} (\widetilde\Omega) \leq \|\varphi^\ast\|B_{2,1}(\Omega)\|(\varphi^{-1})^\ast\|$

Let us consider the H\"older cusp domain $\Omega_{\alpha}\subset\mathbb R^2$ which is the image of the square $\Omega = \{(x,y): |x|+|y| \leq 1\}$ under
$Q$-homeomorphism $\varphi: \Omega \to \Omega_{\alpha}$ of the form
$$
\varphi(x,y)=
\begin{cases}
(x,y^\alpha),& \,\, y \geq 0,\\
(x, -y^\alpha),& \,\, y < 0.
\end{cases}
$$

\begin{thm}
Let $\Omega_{\alpha}\subset\mathbb R^2$ be the H\"older cusp domain. Then 
$$
\mu_1(\Omega_{\alpha}) \geq \frac{(\alpha+1)\alpha(2-\alpha)(3-\alpha)}{9\pi^3\alpha(\alpha+1 + \alpha(2-\alpha)(3-\alpha))}.
$$
In particular, if $\alpha = 4 $, then  $\mu_1(\Omega_{\alpha}) \geq (11,7\pi^3)^{-1}$.
\end{thm}

\begin{proof}
First of all, we calculate the constant in inequality \eqref{PI-Lip} for the square $\Omega = \{(x,y): |x|+|y| \leq 1\}$.
As a bi-Lipschitz mapping, let us consider the radial transformation $\psi: D(0,1) \to \Omega$, acting by the rule 
$$
\psi(x,y)=(l(x,y)x,l(x,y)y), \quad l(x,y)=\frac{\sqrt{x^2+y^2}}{|x|+|y|}.
$$

By simple calculation, we obtain 
$$
|D\psi(x,y)| = \left(\frac{2(x^2+y^2)^2}{(|x|+|y|)^4} + \frac{x^2+y^2}{(|x|+|y|)^2}\right)^{\frac{1}{2}}
$$ 
and 
$$
J(x,y, \psi) = \frac{x^2+y^2}{(|x|+|y|)^2}.
$$
Then, for $(x,y) \in D(0,1)$, 
$$
\sup\limits_{(x,y)} |J(x,y,\psi)|^{\frac{1}{2}} = 1, \,\,\, \sup\limits_{(x,y)}\frac{|D\psi(x,y)|}{J(x,y, \psi)} = 2.
$$
Substituting these values to inequality \eqref{PI-Lip} with bi-Lipschitz mapping $\psi$ we obtain

\begin{equation}\label{PI-Qube}
\left( \iint\limits_{\Omega} |f(x,y)-f_{\Omega}|^2 \, dxdy \right)^{\frac{1}{2}} \leq \frac{3\sqrt{\pi^3}}{2} \iint\limits_{\Omega} |\nabla f(x,y)| \, dxdy,
\end{equation}
for any $f\in W^1_1(\Omega)$.

The next step is to calculate the norms of composition operators
$$
\varphi^{\ast} : L^1_2(\Omega_{\alpha}) \to L^1_1(\Omega) \quad \text{and} \quad (\varphi^{-1})^{\ast} : L_2(\Omega) \to L_2(\Omega_{\alpha}).
$$

By Theorem \ref{CompTh} the norm 
$$
\|\varphi^\ast\| \leq K_{2,2}(\Omega; \varphi) = \left \|\frac{|D\varphi(\cdot)|}{|J(\cdot,\varphi)|^{\frac{1}{2}}} \mid L_2(\Omega) \right \|.
$$ 
In the \cite{VU04}, it was proven, that the norm $(\varphi^{-1})^{\ast}$ can be calculate as 
$$
\|(\varphi^{-1})^{\ast}\| = M_2(\Omega) = \sup\limits_{(x,y) \in \Omega} |J(x,y, \varphi)|^{\frac{1}{2}}.
$$

For the $Q$-homeomorphism $\varphi: \Omega_{\alpha} \to \Omega$,
\begin{align*}
&|D\varphi(x,y)| = \sqrt{1 + \alpha^2 y^{2(\alpha-1)}}, \quad J(x,y, \varphi) = \alpha y^{\alpha-1};\\
&K_{2,2}(\Omega) = \left( 4\int\limits^1\limits_0 dx \int\limits^x\limits_0 \frac{1 + \alpha^2 y^{2(\alpha-1)}}{\alpha y^{\alpha-1}} \, dy \right)^{\frac{1}{2}} = \left( \frac{4}{\alpha(2-\alpha)(3-\alpha)} + \frac{4}{\alpha+1} \right)^{\frac{1}{2}};\\
&M_2(\Omega) = \sup\limits_{(x,y) \in \Omega} |J(x,y,\varphi)|^{\frac{1}{2}} = \sqrt{\alpha}.
\end{align*}

Combine all the estimates, we infer that
$$
\mu_1(\Omega_{\alpha}) \geq \frac{(\alpha+1)\alpha(2-\alpha)(3-\alpha)}{9\pi^3\alpha(\alpha+1 + \alpha(2-\alpha)(3-\alpha))}.
$$

\end{proof}

\vskip 0.3cm

Alexander Menovschikov; Department of Mathematics, University of Hradec Kr\'alov\'e, Rokitansk\'eho 62, 500 03 Hradec Kr\'alov\'e, Czech Republic.
              Faculty of Economics, University of South Bohemia, Studentsk\'a 13, 370 05 \v{C}esk\'e Bud\v{e}jovice, Czech Republic.
 
\emph{E-mail address:} \email{alexander.menovschikov@uhk.cz} \\

Alexander Ukhlov; Department of Mathematics, Ben-Gurion University of the Negev, P.O.Box 653, Beer Sheva, 8410501, Israel 
							
\emph{E-mail address:} \email{ukhlov@math.bgu.ac.il}


\begin{thebibliography}{References}

\bibitem{CHM10} M.~Cs\"ornyei, S.~Hencl and J.~Mal\'y, Homeomorphisms in the Sobolev space $W^{1,n-1}$, J. Reine Angew. Math. , 644 (2010), 221-235.

\bibitem{F69} H.~Federer, Geometric measure theory, Sp\-rin\-ger Verlag, Berlin, 1969.

\bibitem{Ger69} F.~W.~Gehring, Lipschitz mappings and the $p$-capacity of rings in $n$-space, Advances in the theory of Riemann surfaces (Proc. Conf., Stony Brook, N.Y., 1969), 175--193. Ann. of Math. Studies, No. 66. Princeton Univ. Press, Princeton, N.J., 1971.

\bibitem{GGu} V.~Gol'dshtein, L.~Gurov, Applications of change of variables operators for exact embedding theorems, 
Integral Equations Operator Theory 19 (1994), 1--24.

\bibitem{GGR95} V.~Gol'dshtein, L.~Gurov, A.~Romanov, Homeomorphisms that induce monomorphisms of Sobolev spaces, 
Israel J. Math., 91 (1995), 31--60.

\bibitem{GPU18_3} V.~Gol'dshtein, V.~Pchelintsev, A.~Ukhlov, On the First Eigenvalue of the Degenerate p-Laplace Operator in Non-convex Domains. Integral Equations Operator Theory 90 (2018), 90:43.

\bibitem{GPU18} V.~Gol'dshtein, V.~Pchelintsev, A.~Ukhlov, 	Integral estimates of conformal derivatives and spectral properties of the Neumann-Laplacian. Journal of Mathematical Analysis and Applications, 463 (2018), 19:39.

\bibitem{GS82} V.~M.~Gol'dshtein, V.~N.~Sitnikov, Continuation of functions of the class $W^1_p$ across H\"older boundaries, Imbedding theorems and their applica-ions, Trudy Sem. S. L. Soboleva, 1 (1982), 31--43.

\bibitem{GU} V.~Gol'dshtein, A.~Ukhlov, Weighted Sobolev spaces and embedding theorems, Trans. Amer. Math. Soc., 361, (2009), 3829--3850.

\bibitem{GU10} V.~Gol'dshtein, A.~Ukhlov, About homeomorphisms that induce composition operators on Sobolev spaces, 
Complex Var. Elliptic Equ., 55 (2010), 833--845.

\bibitem{GU14} V.~Gol'dshtein, A.~Ukhlov, Brennan's Conjecture and universal Sobolev inequalities, Bull. Sci. Math., 138 (2014), 253--269.

\bibitem{GU16} V.~Gol'dshtein, A.~Ukhlov, On the first Eigenvalues of Free Vibrating Membranes in Conformal Regular Domains, 
Arch. Rational Mech. Anal., 221 (2016), no. 2, 893--915.

\bibitem{GU16_2} V.~Gol'dshtein, A.~Ukhlov, Traces of function of $L^1_2$  Dirichlet spaces on the Caratheodory boundary,
Studia Math. 235 (2016), 209--224. 

\bibitem{GU17} V.~Gol'dshtein, A.~Ukhlov, The spectral estimates for the Neumann-Laplace operator in space domains, 
Adv. in Math., 315 (2017), 166--193.

\bibitem{GU19} V.~Gol'dshtein, A.~Ukhlov, On functional properties of weak $(p,q)$-quasiconformal homeomorphisms,
Ukrainian Math. Bull., 16 (2019), 329--344.

\bibitem{H93}  P.~Hajlasz, Change of variables formula under minimal assumptions, Colloq. Math., 64 (1993), 93--101.

\bibitem{H50}  P.~Halmos, Measure Theory, Springer-Verlag New York, 1950.

\bibitem{HeK95} J.~Heinonen, P.~Koskela, Weighted Sobolev and Poincar\'e inequalities and quasiregular mappings of polynomial type, 
Math. Scand. 77 (1995), 251--271. 

\bibitem{HK14} S.~Hencl, P.~Koskela, Lectures on Mappings of Finite Distortion, Springer, Berlin/Heidelberg, 2014.

\bibitem{KZ20} P.~Koskela, Zh.~Zhu, Sobolev extensions via reflections,  	arXiv:1812.09037.

\bibitem{K86} V.~I.~Kruglikov, Capacities of condensors and quasiconformal in the mean mappings in space. Mat. Sb. (N.S.) 130 (172) (1986), 185--206. 

\bibitem{Kr64} S.~L.~Krushkal, On mean quasiconformal mappings, Dokl. Akad. Nauk SSSR, 157 (1964), 517--519.

\bibitem{MRSY01} O.~Martio, V.~Ryazanov, U.~Srebro, E.~Yakubov, To the the\-ory of $Q$-ho\-me\-o\-morphisms, Dokl. Akad. Nauk Rossii, 381 (2001), 20--22.

\bibitem{MRSY04} O.~Martio, V.~Ryazanov, U.~Srebro and E.~Yakubov,  Mappings with finite length distortion,  J. d'Anal. Math.,
93 (2004), 215--236.

\bibitem{MRSY09} O.~Martio, V.~Ryazanov, U.~Srebro, E.~Yakubov, Moduli in modern mapping theory. Springer Monographs in Mathematics. Springer, New York, 2009.

\bibitem{M69} V.~G.~Maz'ya, Weak solutions of the Dirichlet and Neumann problems,
Trudy Moskov. Mat. Ob-va., 20 (1969), 137--172 (1969)

\bibitem{M} V.~Maz'ya, Sobolev spaces: with applications to elliptic partial differential equations, Springer, Berlin/Heidelberg, 2010.

\bibitem{MH72} V.~G.~Maz'ya, V.~P.~Havin, Non-linear potential theory, Russian Math. Surveys, \textbf{27} (1972), 71--148.

\bibitem{MU21} A.~Menovschikov, A.~Ukhlov, Composition operators on Hardy-Sobolev spaces and $\BMO$-quasiconformal mappings, Ukrainian Math. Bull., (2021) 

\bibitem{PW} L.~E.~Payne, H.~F.~Weinberger, An optimal Poincar\'e inequality for convex domains, Arch. Rat. Mech. Anal., 5 (1960), 286--292.

\bibitem{P69} I.~N.~Pesin, Mappings that are quasiconformal in the mean, Soviet Math. Dokl., 10 (1969), 939--941.

\bibitem{R96} V.~I.~Ryazanov, On mean quasiconformal mappings, Siberian Math. J., 37 (1996), 325--334.

\bibitem{Sl41} S.~L.~Sobolev, On some groups of transformations of  the $n$-dimensional space, Dokl. AN SSSR, 32 (1941), 380--382.

\bibitem{S85} G.~D.~Suvorov, The generalized "length and area principle" in mapping theory, Naukova Dumka, Kiev, 1985. 

\bibitem{U93} A.~Ukhlov, On mappings, which induce embeddings of Sobolev spaces, Siberian Math. J., 34 (1993), 185--192.

\bibitem{UV08} A.~Ukhlov, S.~K.~Vodop'yanov, Mappings associated with weighted Sobolev spaces. Complex analysis and dynamical systems III, Contemp. Math., Amer. Math. Soc., Providence, RI. 455 (2008), 369--382.

\bibitem{V71} J.~V\"ais\"al\"a, Lectures on $n$-dimensional quasiconformal mappings. Lecture Notes in Math. 229, Springer Verlag, Berlin, 1971.

\bibitem{V88} S.~K.~Vodop'yanov, Taylor Formula and Function Spaces, Novosibirsk Univ. Press., 1988.

\bibitem{V20} S.~K.~Vodop'yanov, Composition Operators on Weighted Sobolev Spaces and the Theory of $Q_p$-Homeomorphisms, Doklady Mathematics, 102 (2020), 371--375.

\bibitem{VGR} S.~K.~Vodop'yanov, V.~M.~Gol'dshtein, Yu.~G.~Reshetnyak, On geometric properties of functions with generalized first derivatives, Uspekhi Mat. Nauk {34} (1979), 17--65. 

\bibitem{VU98} S.~K.~Vodop'yanov, A.~D.~Ukhlov, Sobolev spaces and $(p,q)$-quasiconformal mappings of Carnot groups. Siberian Math. J., 39 (1998), 776--795. 

\bibitem{VU02} S.~K.~Vodop'yanov, A.~D.~Ukhlov, Superposition operators in Sobolev spaces, Russian Mathematics (Izvestiya VUZ) 46 (2002), no. 4, 11--33. 

\bibitem{VU04} S.~K.~Vodop'yanov, A.~D.~Ukhlov, Set Functions and Their Applications in the Theory of Lebesgue and Sobolev Spaces. I, Siberian Adv. Math., 14:4 (2004), 78--125.


\end{thebibliography}
\end{document}